\documentclass[a4paper,12pt]{amsart}
\usepackage{amsmath}
\usepackage{amsthm}
\usepackage{amssymb}
\usepackage{graphicx}
\numberwithin{equation}{section}
\theoremstyle{plain}
\newtheorem{theorem}{Theorem}[section]
\newtheorem{proposition}[theorem]{Proposition}
\newtheorem{lemma}[theorem]{Lemma}
\newtheorem{corollary}[theorem]{Corollary}
\theoremstyle{definition}
\newtheorem{definition}[theorem]{Definition}
\newtheorem{remark}[theorem]{Remark}
\newtheorem{problem}{Problem}
\makeatletter
\renewenvironment{abstract}{%
  \ifx\maketitle\relax
    \ClassWarning{\@classname}{Abstract should precede
      \protect\maketitle\space in AMS document classes; reported}%
  \fi
  \global\setbox\abstractbox=\vtop \bgroup
    \normalfont\Small
    \list{}{\labelwidth\z@
      \leftmargin2.95pc \rightmargin\leftmargin
      \listparindent\normalparindent \itemindent\z@
      \parsep\z@ \@plus\p@
      
    }%
    \item[\hskip\labelsep\scshape\abstractname.]%
}{%
  \endlist\egroup
  \ifx\@setabstract\relax \@setabstracta \fi
}
\makeatother
\newcommand{\defeq}{\mathrel{\mathop:}=}
\newcommand{\defqe}{=\mathrel{\mathop:}}
\newcommand{\R}{{\mathbb R}}
\newcommand{\C}{{\mathbb C}}

\newcommand{\RR}{{\mathcal R}}
\renewcommand{\i}{{i\,}}
\newcommand{\vg}{{v_g}}
\newcommand{\X}{{\mathfrak X}}
\DeclareMathOperator{\id}{id}
\DeclareMathOperator{\Div}{div}
\DeclareMathOperator{\arctanh}{arcth}
\newcommand{\sn}{{\mathbf{sn}}}
\newcommand{\cn}{{\mathbf{cn}}}
\newcommand{\dn}{{\mathbf{dn}}}
\newcommand{\am}{{\mathbf{am}}}
\renewcommand{\emph}[1]{{\itshape #1}}
\begin{document}
\title[Conformal change of Riemannian metrics and biharmonic maps]{Conformal change of Riemannian metrics and biharmonic maps}
\author{Hisashi Naito}
\address{
  Graduate School of Mathematics, 
  Nagoya University, 
  \newline\hspace{\parskip}
  Furocho, Chikusaku, Nagoya, 464-8602, Japan
}
\email{naito@math.nagoya-u.ac.jp}
\author{Hajime Urakawa}
\address{
  Institute for International Education, Tohoku University, 
  \newline\hspace{\parskip}
  Kawauchi 41, Sendai, 980-8576, Japan
}
\email{urakawa@math.is.tohoku.ac.jp}
\thanks{
  Supported by the Grant-in-Aid for the Scientific Research, (C) No. 21540207 and (C) No. 22540075, 
  Japan Society for the Promotion of Science. 
}
\begin{abstract}
  For the reduction ordinary differential equation due to Baird and Kamissoko \cite{BK} for biharmonic maps from a Riemannian manifold
  $(M^m,g)$ into another one  $(N^n,h)$, we show that this ODE has no global positive solution for every $m\geq 5$.
  On the contrary, we show that 
  there exist global positive solutions in the case $m=3$. 
  As applications, for the the Riemannian product $(M^3,g)$ of the line and a Riemann surface, we construct the new metric $\widetilde{g}$ on $M^3$ conformal to $g$ 
  such that every nontrivial product harmonic map from $M^3$ with respect to  the original metric $g$ must be biharmonic but not harmonic with respect to the new metric $\widetilde{g}$. 
\end{abstract}
\maketitle
\section{Introduction}
Harmonic maps play a central role in variational problems and geometry.
They are critical points of the energy functional 
$E(\varphi)=\frac12\int_M \Vert d\varphi\Vert^2 \vg$
for smooth maps $\varphi$ of $(M,g)$ into $(N,h)$, 
and the Euler-Lagrange equation is that the tension field $\tau(\varphi)$ vanishes. 
By extending the notion of harmonic maps, in 1983, J. Eells and L. Lemaire \cite{EL1} 
introduced the bienergy functional
\begin{displaymath}
  E_2(\varphi)=\frac12\int_M \vert\tau(\varphi)\vert^2\vg.
\end{displaymath}
After G.Y. Jiang \cite{J} studied the first and second variation formulas of $E_2$, whose critical maps are called biharmonic maps, 
there have been extensive studies in this area 
(for instance, see \cite{CMP}, \cite{I}, \cite{II}, \cite{LO}, \cite{LO2}, \cite{MO1}, \cite{O1}, etc.). 
It is clear that the biharmonicity derives from harmonicity. 
In the compact case, a harmonic map is a minimum of the bienergy functional, therefore a critical point, 
i.e., it is biharmonic. In the non-compact case, from the biharmonic equation 
$\tau(\varphi)=0$ implies $\tau_2(\varphi)=0$. 
\par 
One of most important problems on biharmonic maps is to 
construct biharmonic maps which are not harmonic (cf. \cite{BK}, \cite{CMP}, \cite{IIU}, \cite{J}, \cite{MO1}, \cite{O1}). 
P. Baird and D. Kamissoko \cite{BK} raised an interesting idea 
to produce a biharmonic but not harmonic map 
of $(M,\widetilde{g})$ into $(N,h)$ 
with conformal change 
of $g$ into $\widetilde{g}$,  by a factor of 
$C^{\infty}$ function $f$, and 
they reduced to problem to the existence of non-trivial positive solutions of the ordinary differential equation on 
$f$:
\begin{equation}
  \label{eq:0}
  f^2 f'''-2 \frac{m+1}{m-2} f f' f''+\frac{m^2}{(m-2)^2} {f'}^3=0, 
\end{equation}
and gave some interesting examples 
biharmonic maps.  However, they did not reach to a final answer on the very interesting and important existence problem of global solutions of the ODE (\ref{eq:0}).
\par
In this paper, we give a final and complete answer to this problem. 
Our conclusion is the following: 
If $\dim M \ge 5$, 
there is NO global solution of this ODE (\ref{eq:0}) (cf. Theorem \ref{theorem:7:1}),
and 
if $\dim M=3$, $4$, 
there exists a global solution $f$ of (\ref{eq:0}) (cf. Theorem \ref{theorem:7:1}). 
Furthermore, there is NO periodic solution of (\ref{eq:0}) (cf. Theorem \ref{theorem:case3:3}).  
Our method of proof is to analyze the solutions of (\ref{eq:0}) 
based on the comparison theorem for the ordinary differential inequality (cf. Lemma \ref{lemma:6:1}).
In case of $\dim M = 8$, 
the ODE (\ref{eq:0}) is related to the Jacobi's elliptic function, 
hence by using analysis on the poles and zeros of the Jacobi's elliptic function 
and the energy equality (cf. Proposition \ref{proposition:6:2} and Lemma \ref{lemma:7:2}), 
we obtain the non-existence result of global solutions of (\ref{eq:0}).
In case of $\dim M \ge 5$, 
by using the energy inequality (Lemma \ref{lemma:7:2}), 
we also obtain the non-existence result of global solutions of (\ref{eq:0}).
\par 
Finally, we show 
\begin{theorem}[cf. Theorem \ref{theorem:9:1}]
  \label{theorem:introduction:1}
  Assume $m = 3$, $4$, 
  for a given harmonic map 
  $\varphi \colon (\Sigma^{m-1},g)\rightarrow (P,h)$, 
  let us define 
  $\widetilde{\varphi} \colon \R \times \Sigma^{m-1}\ni (x,y)\mapsto 
  (ax+b,\varphi(y))\in {\mathbb R}\times P$ where 
  $a$ and $b$ are constants, and define also 
  $\widetilde{f}(x,y) \defeq f(x)$ ($(x,y)\in \R \times \Sigma^{m-1}$), where 
  $f$ is the solution of the ODE (\ref{eq:0}). 
  Then, 
  \begin{enumerate}
  \item 
    In the case $m=3$, the mapping $\widetilde{\varphi} \colon (\R\times \Sigma^2,\widetilde{f}^2 g) \rightarrow (\R\times P,h)$ is biharmonic, 
    but not harmonic if $a\not=0$.
  \item 
    In the case $m=4$, the mapping  
    $\widetilde{\varphi} \colon (\R \times \Sigma^3,\frac{1}{\cosh x} g)
    \rightarrow (\R\times P,h)$ is biharmonic, 
    but not harmonic if $a\not=0$.   
  \end{enumerate}
\end{theorem} 
Outline of this paper is as follows: 
\par
After preparing the basic materials, we show 
several formulas under conformal change of Riemannian metrics, 
and derive (\ref{eq:0}) by a different manner as 
P. Baird and D. Kamissoko (\cite{BK}) in Sections 3, 4 and 5. 
In Section 5, 
we also construct an 2nd order non-linear ODE (\ref{eq:5:6}) from (\ref{eq:0}) 
by using the Cole-Hopf transformation $u=f'/f$.
In Section 6, we show our main tools to analyze the solutions of (\ref{eq:5:6}), 
the comparison theorem of the ordinary differential inequalities (cf. Lemma \ref{lemma:6:1}), 
the analysis of the poles and zeros of the Jacobi's elliptic functions (cf. Proposition \ref{proposition:6:2}), 
and the energy inequality for (\ref{eq:5:6}) (cf. Lemma \ref{lemma:7:2}).
In Section 7, we show main results of this paper, 
the non-existence and existence of solutions of (\ref{eq:5:6}) and hence them of (\ref{eq:0})
by using the energy inequality and the comparison theorem.
Finally, in Section 8, we show applications to constructions of biharmonic, but not harmonic maps between the product manifolds whose 
dimension of the domain manifold $M$ is 3 or 4 
(cf. Theorem 1.1, Theorem \ref{theorem:9:1}). 
\section{Preliminaries}
In this section, we prepare materials for the first and second variation formulas for the bienergy functional and biharmonic maps. 
Let us recall the definition of a harmonic map $\varphi\colon (M,g)\rightarrow (N,h)$, of a compact Riemannian manifold $(M,g)$ into another Riemannian manifold $(N,h)$, 
which is an extremal of the {\em energy functional} defined by 
\begin{displaymath}
  E(\varphi)=\int_Me(\varphi)\vg
\end{displaymath}
where $e(\varphi) \defeq \frac12\vert d\varphi\vert^2$ is called the energy density of $\varphi$. 
That is, for any variation $\{\varphi_t\}$ of $\varphi$ with $\varphi_0=\varphi$, 
\begin{equation}
  \label{eq:preliminaries:1}
  \left.
    \frac{d}{dt}
  \right|_{t=0}
  E(\varphi_t)
  =
  -\int_{M}h(\tau(\varphi),V)\vg
  =0,
\end{equation}
where $V\in \Gamma(\varphi^{-1}TN)$ 
is a variation vector field along $\varphi$ which is given by 
\begin{math}
 V(x)=\left.\frac{d}{dt}\right|_{t=0}\varphi_t(x)\in T_{\varphi(x)}N, 
\end{math}
$(x\in M)$, 
and the {\em tension field} is given by 
\begin{math}
  \tau(\varphi)
  =
  \sum_{i=1}^m B(\varphi)(e_i,e_i)\in \Gamma(\varphi^{-1}TN), 
\end{math}
where $\{e_i\}_{i=1}^m$ is a locally defined frame field on $(M,g)$, 
and $B(\varphi)$ is the second fundamental form of $\varphi$ defined by 
\begin{equation}
  \label{eq:preliminaries:2}
  \begin{aligned}
    B(\varphi)(X,Y)
    &=(\widetilde{\nabla}d\varphi)(X,Y)
    \\
    &=(\widetilde{\nabla}_X d\varphi)(Y)
    \\
    &=\overline{\nabla}_X(d\varphi(Y))-d\varphi(\nabla_X Y)
    \\
    &={}^N\!\nabla_{\varphi_{\ast}(X)}d\varphi(Y) -\varphi_{\ast}(\nabla_X Y),
  \end{aligned}
\end{equation}
for all vector fields $X$, $Y \in \X(M)$. 
Furthermore, $\nabla$, and ${}^N\!\nabla$, are connections on $TM$, $TN$ of $(M,g)$, $(N,h)$, respectively, 
and $\overline{\nabla}$, 
and $\widetilde{\nabla}$ 
are the induced ones on $\varphi^{-1}TN$, 
and $T^{\ast} M\otimes \varphi^{-1}TN$, respectively. 
By (\ref{eq:preliminaries:1}), $\varphi$ is harmonic if and only if $\tau(\varphi)=0$. 
\par
The second variation formula is given as follows. 
Assume that $\varphi$ is harmonic. 
Then, 
\begin{equation}
  \label{eq:preliminaries:3}
  \left.
    \frac{d^2}{dt^2}
  \right|_{t=0}
  E(\varphi_t)
  =\int_{M}h(J(V),V)\vg
\end{equation}
where $J$ is an elliptic differential operator, called 
{\em Jacobi operator} defined on $\Gamma(\varphi^{-1}TN)$ given by 
\begin{equation}
  \label{eq:preliminaries:4}
  J(V)=\overline{\Delta}V-\RR(V),
\end{equation}
where 
$\overline{\Delta}V=\overline{\nabla}^{\ast}\overline{\nabla}V$ 
is the {\em rough Laplacian} and 
$\RR$ is a linear operator on $\Gamma(\varphi^{-1}TN)$
given by 
\begin{math}
  \RR V=
  \sum_{i=1}^m R^{\!N}(V,d\varphi(e_i))d\varphi(e_i), 
\end{math}
and $R^{\!N}$ is the curvature tensor of $(N,h)$ given by 
\begin{math}
  R^{\!N}(U,V)W={}^N\!\nabla_U{}^N\!\nabla_V W
  -{}^N\!\nabla_V{}^N\!\nabla_U W-{}^N\!\nabla_{[U,V]}W
\end{math}
for $U$, $V$, $W \in \X(N)$.   
\par
J. Eells and L. Lemaire proposed (\cite{EL1}) polyharmonic ($k$-harmonic) maps and 
Jiang studied (\cite{J}) the first and second variation formulas of biharmonic maps. Let us consider the {\em bienergy functional} 
defined by 
\begin{equation}
  \label{eq:preliminaries:5}
  E_2(\varphi)=\frac12\int_M\vert\tau(\varphi)\vert ^2\vg
\end{equation}
where $\vert V\vert^2=h(V,V)$, $V\in \Gamma(\varphi^{-1}TN)$.  
Then, the first and second variation formulas are given as follows. 
\begin{theorem}[The first variation formula \cite{J}]
  \label{theorem:preliminaries:1}
  \begin{equation}
    \label{eq:preliminaries:6}
    \left.
      \frac{d}{dt}
    \right|_{t=0}
    E_2(\varphi_t)
    =-\int_{M}h(\tau_2(\varphi),V)\vg
  \end{equation}
  where 
  \begin{equation}
    \label{eq:preliminaries:7}
    \tau_2(\varphi)=J(\tau(\varphi))=\overline{\Delta}\tau(\varphi)-\RR(\tau(\varphi)),
  \end{equation}
  $J$ is given in (\ref{eq:preliminaries:4}).
\end{theorem}
\begin{definition}
  \label{definition:preliminaries:2}
  A smooth map $\varphi$ of $M$ into $N$ is said to be 
  {\em biharmonic} if 
  $\tau_2(\varphi)=0$. 
\end{definition} 
\section{Formulas under conformal change of Riemannian metrics}
In this section, 
we show several formulas under conformal changes of the domain Riemannian manifold $(M,g)$. 
Let $(M,g)$ and $(N,h)$ be two Riemannian manifolds, and $\varphi\colon M\rightarrow N$, a $C^{\infty}$ map from $M$ into $N$. 
We denote the {\it energy functional} by 
\begin{equation}
  \label{eq:3:1}
  E_1(\varphi:g,h) \defeq \frac12
  \int_M\vert d\varphi\vert_{g,h}{}^2\vg
\end{equation}
where $\vert d\varphi\vert_{g,h}^2$ 
is twice of the energy density of $\varphi$, i.e., 
\begin{displaymath}
  \vert d\varphi\vert_{g,h}{}^2
  \defeq
  \sum_{i=1}^m h(\varphi_{\ast}e_i,\varphi_{\ast}e_i),
\end{displaymath}
and $\{e_i\}_{i=1}^m$ is a local orthonormal frame field on $(M,g)$. 
For a positive $C^{\infty}$ function on $M$, 
we consider the conformal change of the Riemannian metric on $M$, 
$\widetilde{g} \defeq f^{2/(m-2)} g$, where $m=\dim M>2$. 
Then, we have (cf. \cite{EF}) that 
\begin{align}
  \vert d\varphi\vert_{\widetilde{g},h}{}^2
  &=f^{-2/(m-2)} \vert d\varphi\vert_{g,h}{}^2,
  \label{eq:3:2}
  \\
  v_{\widetilde{g}}&=f^{m/(m-2)}\vg.
  \label{eq:3:3}
\end{align}
Thus, we have (\cite[p.161]{EF}) that
\begin{equation}
  \label{eq:3:4}
  E_1(\varphi:\widetilde{g},h)
  =\frac12\int_M f \vert d\varphi\vert_{g,h}{}^2\vg.
\end{equation}
\par
Let us consider the {\it bienergy functional} defined by 
\begin{equation}
  \label{eq:3:5}
  E_2(\varphi:g,h) \defeq \frac12\int_M\vert\tau_g(\varphi)\vert_{g,h}{}^2\vg
\end{equation}
where 
\begin{equation}
  \label{eq:3:6}
  \tau_g(\varphi) \defeq \sum_{i=1}^m
  \left\{{}^{N}\!\nabla_{\varphi_{\ast}e_i}\varphi_{\ast}e_i-\varphi_{\ast}(\nabla^g_{e_i}e_i)
  \right\}\in \Gamma(\varphi^{-1}TN)),
\end{equation}
${}^{N}\!\nabla$, $\nabla^g$ are the Levi-Civita connections of $(N,h)$, $(M,g)$, respectively.  
\par
We first see that 
\begin{equation}
  \label{eq:3:7}
  \begin{aligned}
    \nabla^{\widetilde{g}}_X Y
    =
    \nabla^g_X Y
    &+\frac{1}{m-2}
    \left\{
      f^{-1}(X f)Y
      +f^{-1}(Y f)X
      \vphantom{-g(X,Y) f^{-1}\sum_{i=1}^m  (e_if)e_i,}
    \right.
    \\
    &
    \left.
      -g(X,Y) f^{-1}\sum_{i=1}^m
      (e_if)e_i,
    \right\}
  \end{aligned}
\end{equation}
for all $X,Y\in \X(M)$. 
Then, we have 
\begin{equation}
  \label{eq:3:8-10}
  \begin{aligned}
    \tau_{\widetilde{g}}(\varphi)
    &=f^{-2/(m-2)} \tau_g(\varphi)
    +f^{-m/(m-2)} \varphi_{\ast}(\nabla^g f)
    \\
    &=f^{2/(2-m)}
    \left\{
      \tau_g(\varphi)+f^{-1}\varphi_{\ast}(\nabla^g f)
    \right\}\\
    &=f^{m/(2-m)} {\Div}_g(f d\varphi),
  \end{aligned}
\end{equation}
where 
$\nabla^g f \defeq \sum_{j=1}^m(e_j f) e_j\in  \X(M)$ 
for $f\in C^{\infty}(M)$, and 
\begin{displaymath}
  \begin{aligned}
    {\Div}_g(d\varphi)& \defeq \sum_{i=1}^m
    (\widetilde{\nabla}_{e_i}d\varphi)(e_i)
    =\sum_{i=1}^m
    \left\{
      \widetilde{\nabla}_{e_i}(d\varphi(e_i))
      -d\varphi(\nabla_{e_i}e_i)
    \right\}\\
    &=\sum_{i=1}^m
    \left\{
      {}^N\!\nabla_{\varphi_{\ast}(e_i)}d\varphi(e_i)
      -\varphi_{\ast}(\nabla_{e_i}e_i)
    \right\}.
  \end{aligned}
\end{displaymath}
Here, recall that 
$\widetilde{\nabla}$ is the induced connection on 
$\varphi^{-1}TN \otimes T^{\ast}M$ from ${}^{N}\!\nabla$ and $\widetilde{g}$, 
and we have 
\begin{equation}
  \label{eq:3:11}
  f^{m/(2-m)} {\Div}_g(f d\varphi)
  =f^{m/(2-m)} d\varphi(\nabla^g f)
  +f^{2/(2-m)}\tau_g(\varphi).
\end{equation}
Therefore, it holds (cf. \cite{EF}) that
$\varphi \colon (M,\widetilde{g})\rightarrow (N,h)$ is harmonic 
if and only if 
\begin{equation}
  \label{eq:3:12}
  f \tau_g(\varphi)+\varphi_{\ast}(\nabla^g f)=0.
\end{equation}
\par
Summing up the above, we have 
\begin{lemma}[cf. {\cite[p.161]{EF}, \cite[p.135]{YLOu}}]
  \label{lemma:3:1}
  The Euler-Lagrange equation of the energy functional 
  $E_1(\varphi:\widetilde{g},h)$ 
  is given by 
  \begin{align}
    \tau(\varphi:\widetilde{g},h)
    &=f^{2/(2-m)} 
    \left\{
      \tau(\varphi:g,h)+\varphi_{\ast}(\nabla^g\log f)
    \right\}
    \label{eq:3:13}
    \\
    &=f^{m/(2-m)} {\Div}_g(f d\varphi).
    \label{eq:3:14}
  \end{align}
  Thus, $\varphi \colon (M,\widetilde{g})\rightarrow (N,h)$ 
  is harmonic if and only if 
  ${\Div}_g(f d\varphi)=0$. 
\end{lemma}
\par
Next, we compute the Euler-Lagrange 
equation of the {\it bienergy functional}: 
\begin{displaymath}
  E_2(\varphi:\widetilde{g},h)=
  \frac12\int_M
  \vert
  \tau_{\widetilde{g}}(\varphi)
  \vert_{\widetilde{g},h}{}^2 v_{\widetilde{g}}.
\end{displaymath}
It is known (cf. \cite{J}) that
\begin{equation}
  \label{eq:3:15}
  \begin{aligned}
    \tau_2(\varphi:\widetilde{g},h)
    =&J_{\widetilde{g}}(\tau_{\widetilde{g}}(\varphi))
    \\
    =&\overline{\Delta}_{\widetilde{g}}(\tau_{\widetilde{g}}(\varphi))-\RR_{\widetilde{g}}(\tau_{\widetilde{g}}(\varphi))
    \\
    =&
    -\sum_{i=1}^m
    \left\{
      \overline{\nabla}_{\widetilde{e_i}}
      (\overline{\nabla}_{\widetilde{e_i}}
      \tau_{\widetilde{g}}(\varphi))
      -\overline{\nabla}
      _{\nabla^{\widetilde{g}}_{\widetilde{e_i}}\widetilde{e_i}}\tau_{\widetilde{g}}(\varphi)
    \right\}
    \\
    &-
    \sum_{i=1}^m
    R^{\!N}(\tau_{\widetilde{g}}(\varphi),
    \varphi_{\ast}\widetilde{e_i})
    \varphi_{\ast}\widetilde{e_i},
  \end{aligned}
\end{equation}
where 
$\overline{\nabla}$ is the induced connection 
on $\varphi^{-1}TN$ from the Levi-Civita connection 
${}^N\!\nabla$ on $TN$ of $(N,h)$, 
and $\{\widetilde{e_i}\}_{i=1}^m$ 
is the local orthonormal frame field on 
$(M,\widetilde{g})$ given by 
$\widetilde{e_i} \defeq f^{-1/(m-2)}e_i$ $(i=1,\cdots,m)$. 
\par
We first calculate 
$J_{\widetilde{g}}(V)$, $(V\in \Gamma(\varphi^{-1}TN))$ given by definition as
\begin{align}
  \label{eq:3:16}
  J_{\widetilde{g}}(V)
  &\defeq
  \overline{\Delta}_{\widetilde{g}}(V)-\RR_{\widetilde{g}}(V)\nonumber\\
  &=
  -\sum_{i=1}^m
  \left\{
    \overline{\nabla}_{\widetilde{e_i}}
    (\overline{\nabla}_{\widetilde{e_i}}
    V)
    -\overline{\nabla}
    _{\nabla^{\widetilde{g}}_{\widetilde{e_i}}\widetilde{e_i}}V
  \right\}
  -
  \sum_{i=1}^m
  R^{\!N}(V,
  \varphi_{\ast}\widetilde{e_i})
  \varphi_{\ast}\widetilde{e_i}.
\end{align}
\begin{lemma}
  \label{lemma:3:2}
  The Jacobi operator with respect to $\widetilde{g}$ is 
  \begin{equation}
    \label{eq:3:17}
    J_{\widetilde{g}}(V)=
    f^{2/(2-m)}J_g(V)-f^{m/(2-m)}\overline{\nabla}_{\nabla^g f}V,\quad (V\in \Gamma(\varphi^{-1}TN)).
  \end{equation}
\end{lemma}
\begin{lemma} [cf. {\cite[pp.135-136]{YLOu}}]
  \label{lemma:3:3}
  For all $f\in C^{\infty}(M)$, 
  $V\in \Gamma(\varphi^{-1}TN)$, real numbers $p$ and $q$, 
  we have
  \begin{align}
    J_g(f V) &= (\Delta_g f) V-2\overline{\nabla}_{\nabla^g f}V+f J_g V,
    \label{eq:3:23}
    \\
    \nabla^g f^p &= p f^{p-1} \nabla^g f, 
    \label{eq:3:24}
    \\
    (\nabla^g f)f^q &= q f^{q-1} \vert \nabla^g f\vert_g{}^2,
    \label{eq:3:25}
    \\
    \Delta_g f^p &= p f^{p-1} \Delta_g f-p(p-1) f^{p-2} \vert\nabla^g f\vert_g{}^2.
    \label{eq:3:26}
  \end{align}
\end{lemma}
\begin{lemma}
  \label{lemma:3:4}
  The bienergy tension field 
  $\tau_2(\varphi:\widetilde{g},h)$ is given by 
  \begin{displaymath}
    \begin{aligned}
      &\tau_2(\varphi:\widetilde{g},h)
      \defeq J_{\widetilde{g}}(\tau_{\widetilde{g}}(\varphi))
      \\
      &=
      \left\{
        -\frac{4}{(2-m)^2} f^{2m/(2-m)} \vert \nabla^g f\vert_g^2
        +\frac{2}{2-m} f^{(2+m)/(2-m)} \Delta_g f
      \right\} 
      \tau_g(\varphi)
      \\
      &-\frac{6-m}{2-m} f^{(2+m)/(2-m)}
       \overline{\nabla}_{\nabla^g f}\tau_g(\varphi)
      +f^{4/(2-m)} J_g(\tau_g(\varphi))
      \\
      &+
      \left\{
        -\frac{m^2}{(2-m)^2} 
        f^{(-2+3m)/(2-m)} \vert\nabla^g f\vert_g^2
        +\frac{m}{2-m} f^{2m/(2-m)} \Delta_g f
      \right\}
       \varphi_{\ast}(\nabla^g f)
      \\
      &-\frac{2+m}{2-m} f^{2m/(2-m)} \overline{\nabla}_{\nabla^g f}\varphi_{\ast}(\nabla^g f)
      +f^{(2+m)/(2-m)} J_g
      (\varphi_{\ast}(\nabla^g f)).
    \end{aligned}
  \end{displaymath}
\end{lemma}
By a direct computation, we have Lemmas \ref{lemma:3:1}, \ref{lemma:3:2}, \ref{lemma:3:3} and \ref{lemma:3:4}.
The proofs are omitted.
Thus, we have immediately 
\begin{corollary} 
  \label{corollary:3:5}
  The bienergy tension field is 
  \begin{displaymath}
    \begin{aligned}
      &f^{2m/(m-2)} \tau_2(\varphi:\widetilde{g},h)
      =
      \left\{
        -\frac{4}{(m-2)^2} \vert \nabla^g f\vert_g{}^2
        -\frac{2}{m-2} f \Delta_g f
      \right\}
      \tau_g(\varphi)
      \\
      &-\frac{m-6}{m-2} f \overline{\nabla}_{\nabla^g f}\tau_g(\varphi) +f^2 J_g(\tau_g(\varphi))
      \\
      &+f^{-1}
      \left\{
        -\frac{m^2}{(m-2)^2} 
        \vert\nabla^g f\vert_g{}^2
        -\frac{m}{m-2} f \Delta_g f
      \right\}
      \varphi_{\ast}(\nabla^g f)
      \\
      &+\frac{m+2}{m-2}
      \overline{\nabla}_{\nabla^g f}\varphi_{\ast}(\nabla^g f)
      +f J_g (\varphi_{\ast}(\nabla^g f)).
    \end{aligned}
  \end{displaymath}
\end{corollary}
Therefore, we have also
\begin{corollary} 
  \label{corollary:3:6}
  $\varphi:(M,\widetilde{g})\rightarrow (N,h)$ 
  is biharmonic if and only if 
  \begin{equation}
    \label{eq:3:30}
    \begin{aligned}
      &\tau_2(\varphi:\widetilde{g},h)
      =0
      \\
      &\Longleftrightarrow
      \\
      &
      \left\{
        -\frac{4}{(m-2)^2} \vert \nabla^g f\vert_g{}^2
        -\frac{2}{m-2} f \Delta_g f
      \right\}
      f\tau_g(\varphi)
      \\
      &\quad
      -\frac{m-6}{m-2} f^2 \overline{\nabla}_{\nabla^g f}\tau_g(\varphi)
      +f^3 J_g(\tau_g(\varphi))
      \\
      &\quad\quad+
      \left\{
        -\frac{m^2}{(m-2)^2} 
        \vert\nabla^g f\vert_g{}^2
        -\frac{m}{m-2} f \Delta_g f
      \right\}
      \varphi_{\ast}(\nabla^g f)
      \\
      &\qquad\quad+\frac{m+2}{m-2} f 
      \overline{\nabla}_{\nabla^g f}\varphi_{\ast}(\nabla^g f)
      +f^2 J_g
      (\varphi_{\ast}(\nabla^g f))
      =0.
    \end{aligned}
  \end{equation}
\end{corollary}
\section{Reduction of constructing proper biharmonic maps}
\label{sec:4}
In this section, we formulate our problem to construct proper biharmonic maps. A biharmonic map is said to be {\itshape proper} if it is not harmonic. 
Let $(M,g)$, $(N,h)$ be two compact Riemannian manifolds. 
In the following we always assume that $m=\dim (M)\leq 3$. 
Eells and Ferreira \cite{EF} showed that, 
\begin{quote}
  \itshape
  for each homotopy class ${\mathcal H}$ in $C^{\infty}(M,N)$, 
  there exist a Riemannian metric $\widetilde{g}$ which is conformal to $g$, 
  and a $C^{\infty}$ map $\varphi\in {\mathcal H}$ such that 
  $\varphi$ is a harmonic map from $(M,\widetilde{g})$ into $(N,h)$.
\end{quote}
We do not assume, in general, that $M$ and $N$ are compact. 
Let us consider the following problem. 
\begin{quote}
  \vspace{-1\baselineskip}
  \begin{problem}
    \label{problem:1}
    For each homotopy class ${\mathcal H}$ in $C^{\infty}(M,N)$, 
    do there exist a Riemannian metric $\widetilde{g}$ which is conformal to $g$, 
    and a $C^{\infty}$ map in ${\mathcal H}$ such that 
    $\varphi \colon (M,\widetilde{g}) \rightarrow (N,h)$ is a {\em proper biharmonic} map, that is, 
    $\tau_2(\varphi,\widetilde{g},h)=0$ and $\tau(\varphi,\widetilde{g},h) \not=0$ ? 
  \end{problem}
\end{quote}
By regarding the above Eells and Ferreira's result, 
we fix a harmonic map $\varphi: (M,g)\rightarrow (N,h)$, that is, $\tau(\varphi)=0$. 
Then, let us consider the following problem: 
\begin{quote}
  \vspace{-1\baselineskip}
  \begin{problem}
    \label{problem:2}
    Does there exist a positive $C^{\infty}$ function $f$ on $M$ such that, 
    for $\widetilde{g}=f^{2/(m-2)} g$, 
    $\varphi \colon (M,\widetilde{g})\rightarrow (N,h)$ is {\em proper biharmonic}, that is, 
    $\tau_2(\varphi,\widetilde{g},h)=0$ and 
    \newline
    $\tau(\varphi,\widetilde{g},h)\not=0$.
  \end{problem}
\end{quote} 
\par
To concern Problem \ref{problem:2}, we have
\begin{theorem} 
  \label{theorem:4:1}
  Assume that $\varphi \colon (M,g)\rightarrow (N,h)$ is harmonic. 
  For a positive $C^{\infty}$ function $f$ on $M$, 
  let us define $\widetilde{g}=f^{2/(m-2)} g$, 
  a Riemannian metric conformal to $g$. 
  Then, 
  \begin{enumerate}
  \item 
    \label{item:theorem:4:1:1}
    $\varphi \colon (M,\widetilde{g})\rightarrow (N,h)$ is harmonic 
    if and only if $\varphi_{\ast}(\nabla^g f)=0$.
  \item 
    \label{item:theorem:4:1:2}
    $\varphi \colon (M,\widetilde{g})\rightarrow (N,h)$ is biharmonic if and only if the following holds: 
    \begin{equation}
      \label{eq:4:1}
      \begin{aligned}
        -&\left\{
          \frac{m^2}{(m-2)^2} \vert\nabla^g f\vert_g{}^2+\frac{m}{m-2} f (\Delta_g f)
        \right\}
         \varphi_{\ast}(\nabla^g f)
        \\
        &+\frac{m+2}{m-2} f \overline{\nabla}_{\nabla^g f} \varphi_{\ast}(\nabla^g f)+f^2 J_g(\varphi_{\ast}(\nabla^g f))=0.
      \end{aligned}
    \end{equation}
  \end{enumerate}
\end{theorem}
\begin{proof}
  For (\ref{item:theorem:4:1:1}),
  due to (\ref{eq:3:13}) or (\ref{eq:3:14}) in Lemma \ref{lemma:3:1}, we have 
  \begin{equation}
    \label{eq:4:2}
    \tau(\varphi:\widetilde{g},h)=f^{2/(2-m)} \tau(\varphi:g,h)+f^{m/(2-m)} \varphi_{\ast}(\nabla^g f),
  \end{equation}
  which implies that (\ref{item:theorem:4:1:1}).
  For (\ref{item:theorem:4:1:2}), 
  in Corollary \ref{corollary:3:6}, substituting $\tau_g(\varphi)=0$ into (\ref{eq:3:30}), we have 
  immediately (\ref{eq:4:1}). 
\end{proof}
We have immediately 
\begin{corollary}
  \label{corollary:4:2}
  Let $\varphi=\id \colon (M,g) \rightarrow (M,g)$ be the identity map. 
  For a positive $C^{\infty}$ function $f$ on $M$, 
  let us define $\widetilde{g}=f^{2/(m-2)} g$. 
  Then, 
  \begin{enumerate}
  \item 
    \label{item:corollary:4:2:1}
    $\varphi=\id \colon (M,\widetilde{g})\rightarrow (M,g)$ is harmonic 
    if and only if $f$ is a constant.
  \item 
    \label{item:corollary:4:2:2}
    $\varphi=\id \colon (M,\widetilde{g})\rightarrow (M,g)$ is biharmonic if and only if 
    \begin{equation}
      \label{eq:4:3}
      \begin{aligned}
        -&\left\{
          \frac{m^2}{(m-2)^2} \vert\nabla^g f\vert_g{}^2+\frac{m}{m-2} f (\Delta_g f)
        \right\}
         \nabla^g f
        \\
        &+\frac{m+2}{m-2} f \overline{\nabla}_{\nabla^g f} \nabla^g f+f^2 J_g(\nabla^g f)=0, 
      \end{aligned}
    \end{equation}
    which is equivalent to 
    \begin{equation}
      \label{eq:4:4}
      \begin{aligned}
        -&\left\{
          \frac{m^2}{(m-2)^2} \vert X \vert_g{}^2+\frac{m}{m-2} f (\Delta_g f)
        \right\}
         X
        \\
        &+\frac{m+2}{m-2} f {\nabla}_X X
        +f^2 (\overline{\Delta}^g(X)-\rho(X))=0,
      \end{aligned}
    \end{equation}
    where $X=\nabla^g f\in \X(M)$, $\rho(X) \defeq 
    \sum_{i=1}^m R^g(X,e_i)e_i$, is 
    the Ricci tensor of $(M,g)$, and 
    $\overline{\Delta}^g(X) \defeq -\sum_{i=1}^m(\nabla^g_{e_i}\nabla^g_{e_i}X-\nabla^g_{\nabla^g_{e_i}e_i}X)$ 
    is the rough Laplacian on $\X(M)$, respectively. 
  \end{enumerate}
\end{corollary}
\begin{proof}
  (\ref{eq:4:3}) and (\ref{eq:4:4}) follow from (\ref{eq:4:1}), 
  and the formula $J_g(V)=\overline{\Delta}^g(V)-\rho(V)$ $(V\in \X(M))$, 
  for the identity map. 
\end{proof}
\section{The identity map of the Euclidean space} 
Let us consider the $m$ dimensional Euclidean space 
$(M,g)=(\R^m,g_0)$ with the standard coordinate $(x_1,\cdots,x_m)$ $(m\geq 3)$.  
In this case, 
let us take a positive $C^{\infty}$ function 
$f=f(x_1,\cdots,x_m)\in C^{\infty}(\R^m)$. 
Let $X=\nabla^g f=\sum_{i=1}^m f_{x_i}\frac{\partial}{\partial x_i}$, 
where we denote 
$f_{x_i}=\frac{\partial f}{\partial x_i}$. 
Then, since
\begin{equation}
  \label{eq:5:1}
  \left\{
    \begin{aligned}
      \rho&=0,\\
      \Delta_g f&=-\sum_{i=1}^m f_{x_ix_i},\\
      \vert X\vert_g{}^2&=\sum_{i=1}^m f_{x_i}{}^2, \\
      \nabla^g_X X&=\sum_{i=1}^m\bigg\{
      \sum_{j=1}^m f_{x_j} f_{x_ix_j}
      \bigg\}\frac{\partial}{\partial x_i},\\ 
      \overline{\Delta}^g(X)&=
      \sum_{j=1}^m \Delta_g(f_{x_j}) \frac{\partial}{\partial x_j}, 
    \end{aligned}
  \right.
\end{equation}
the equation (\ref{eq:4:3}) is reduced to the following:
\begin{equation}
  \label{eq:5:2}
  \begin{aligned}
    -&\left\{
      \frac{m^2}{(m-2)^2} \sum_{i=1}^m f_{x_i}{}^2
      +\frac{m}{m-2} f (\Delta_g f) f_{x_j}
    \right\}
    \\
    &+f^2 \Delta_g(f_{x_j})+\frac{m+2}{m-2} 
    \sum_{i=1}^m f_{x_i} f_{x_ix_j}=0
    \quad (\forall  j=1,\cdots,m).
  \end{aligned}
\end{equation}
If we consider 
$f=f(x_1,\cdots,x_m)=f(x)$, $x=x_1$, then, 
the equation (\ref{eq:5:2}) is equivalent to 
the following ODE: 
\begin{equation}
  \tag{\ref{eq:0}}
  f^2 f'''-2 \frac{m+1}{m-2} f f' f''
  +\frac{m^2}{(m-2)^2} (f')^3=0.
\end{equation}
In the cases $m=3$, $m=4$, 
(\ref{eq:0}) becomes 
\begin{alignat}{3}
  f^2 f'''-8 f f' f''+9 (f')^3&=0&\qquad&(m=3),
  \label{eq:5:4}
  \\
  f^2 f'''-5 f f' f''+4 (f')^3&=0&\qquad&(m=4).
  \label{eq:5:5}
\end{alignat}
\par
Our problem is reduced to find a positive $C^{\infty}$ solution of (\ref{eq:0}). 
In order to analyze (\ref{eq:0}), we put  $u=f'/f$, 
then the equation (\ref{eq:0}) is reduced to the equation:
\begin{equation}
  \label{eq:5:6}
  u''+\frac{m-8}{m-2} u u'-\frac{2(m-4)}{(m-2)^2} u^3=0. 
\end{equation}
Then, we obtain immediately
\begin{proposition}
  \label{proposition:5:1}
  If a positive $C^{\infty}$ solution $f$ of (\ref{eq:0}) on $\R$, 
  then 
  $u= f'/f$ satisfies (\ref{eq:5:6}). 
  Conversely, for every $C^{\infty}$ solution  
  $u$ of (\ref{eq:5:6}) on $\R$, then 
  \begin{math}
    f(t) = C \exp\left(\int^t u(s) \,ds\right)
  \end{math}
  is a positive solution of (\ref{eq:0})
  for every positive constant $C$. 
\end{proposition}
\section{Behavior of solutions of the ODE}
Due to Proposition \ref{proposition:5:1}, our problem is reduced to 
analyse (\ref{eq:5:6}). To do it, we need 
the following two lemmas. 
\begin{lemma}[Comparison theorem {\cite[Theorem III.4.1]{Hartman}}]
  \label{lemma:6:1}
  Assume that a real valued function 
  $F$ on $\R$ satisfies the Lipshitz condition,  
  i.e., there exists a positive number $L>0$ such that 
  \begin{math}
    \vert F(p)-F(q)\vert\leq L\vert p-q\vert,
  \end{math}
  we have
  \begin{enumerate}
  \item 
    \label{item:lemma:6:1:1}
    Two real valued functions $u$ and $v$ defined on the interval 
    $[0,\epsilon)$ for some 
    positive number $\epsilon>0$ satisfy that 
    \begin{displaymath}
      \left\{
        \begin{aligned}
          u'(t)&\geq F(u(t)),\\
          v'(t)&=F(v(t)),\\
          u(0)&=v(0),
        \end{aligned} 
      \right.
    \end{displaymath}
    then it holds that 
    \begin{math}
      u(t)\geq v(t)
    \end{math}
    for any $t > 0$.
  \item 
    \label{item:lemma:6:1:2}
    Conversely, if $u$ and $v$ satisfy that 
    \begin{displaymath}
      \left\{
        \begin{aligned}
          u'(t)&\leq F(u(t)),\\
          v'(t)&=F(v(t)),\\
          u(0)&=v(0),
        \end{aligned} 
      \right.
    \end{displaymath}
    then it holds that 
    \begin{math}
      u(t)\leq v(t)
    \end{math}
    for any $t > 0$.
  \end{enumerate}
\end{lemma}
Next, we have to prepare Jacobi's $\sn$-function. 
\begin{proposition}
  \label{proposition:6:2}
  \mbox{}\par
  \begin{enumerate}
  \item 
    \label{item:proposition:6:2:1}
    The solution of the initial value problem of the ordinary differential equation 
    \begin{equation}
      \label{eq:6:8}
      (y')^2=1-y^4,  y(0)=0,  y'(0)>0
    \end{equation}
    is given by the Jacobi's elliptic function 
    $y(t)=\sn(\i ,t)$.
  \item 
    \label{item:proposition:6:2:2}
    The function 
    $y(t)=\sn(\i,t)$ is real valued in $t\in \R$, 
    and pure imaginary valued in $t\in \i\R$.
  \item 
    \label{item:proposition:6:2:3}
    The function $y(t)$ is a double periodic function in the whole complex plane $\C$ with the two periods $4K$ and $2\i K'$,  
    where $K>0$ and $K'>0$ are given by 
    \begin{align}
      K &\defeq K(k)=\int^{\pi/2}_0\frac{dx}{\sqrt{1-k^2 \sin^2x}},
      \label{eq:6:9}\\
      K'&\defeq K(k'),\quad 
      k'\defeq \sqrt{1-k^2},
      \label{eq:6:10}
    \end{align}
    and it has the only one zeros at $2n K+2m\i K'$, 
    and has the only poles at $2n K+(2m+1)\i K'$, where $m$ and $n$ run over the set of all integers.
  \item 
    \label{item:proposition:6:2:4}
    In particular, 
    it has no pole in the real axis, and has no zero on the imaginary axis except $0$. 
    Furthermore, it has poles on the two lines through the origin with angles 
    $\pm\pi/4$ in the complex plane $\C$.
  \end{enumerate}
\end{proposition}
\begin{proof}
  For (\ref{item:proposition:6:2:1})
  we have to see the function $y(t)=\sn(\i,t)$ solves (\ref{eq:6:8}).   
  Let us recall (cf. 281 Elliptic Functions, \cite[pp. 873--876]{M})
  the elliptic integral of the first kind 
  $u(k,\varphi)$ with modulus $k$ given by 
  \begin{equation}
    \label{eq:6:11}
    u(k,\varphi)=\int^{\varphi}_0 
    \frac{d\psi}{\sqrt{1-k^2 \sin^2\psi}}, 
  \end{equation}
  and its inverse function is the amplitude function 
  $\varphi=\am(k,u)$. By differentiating (\ref{eq:6:11}), 
  we have 
  \begin{equation}
    \label{eq:6:12}
    \frac{d}{d\varphi}u(k,\varphi)
    =\frac{1}{\sqrt{1-k^2 \sin^2\varphi}}. 
  \end{equation}
  Then, we have 
  \begin{equation}
    \label{eq:6:13}
    \sn(k,t)=\sin(\am(k,t)), 
  \end{equation}
  which implies immediately that  
  \begin{equation}
    \label{eq:6:14}
    \begin{aligned}
      \frac{d}{dt}\sn(k,t)
      &=
      \left(
        \frac{d}{dt}\am(k,t)
      \right)
      \cos (\am(k,t))
      \\
      &=
      \left(
        \frac{d}{dt}\am(k,t)
      \right)
      \sqrt{1-\sin^2(\am(k,t))}
      \\
      &=
      \sqrt{1-k^2\sin^2(\am(k,t))}
      \sqrt{1-\sn^2(k,t)}
      \\
      &=
      \sqrt{(1-k^2\sn^2(k,t))(1-\sn^2(k,t))}.
    \end{aligned}
  \end{equation}
  Here, we put $k=\i = \sqrt{-1}$ in (\ref{eq:6:14}), we have 
  \begin{equation}
    \label{eq:6:15}
    \frac{d}{dt}\sn(\i,t) = \sqrt{1-\sn^4(\i,t)},
  \end{equation}
  that is, the function 
  $y(t)=\sn(\i,t)$ is a solution of 
  the differential equation of (\ref{eq:6:8}).
  Since
  $\sn(\i,0)=\sin(\am(\i,0))$ and $\am(\i,0)=0$, 
  we have $\sn(\i,0)=0$. 
  \par
  To get $y'(0)>0$, we only notice that, if we denote as the usual manner 
  \begin{displaymath}
    \begin{aligned}
      \cn(k,t)&\defeq \cos (\am(k,t))), \\
      \dn(k,t)&\defeq \sqrt{1-k^2 \sn^2(k,t)},
    \end{aligned}
  \end{displaymath}
  it holds that 
  \begin{align}
    &\frac{d}{dt}\sn(k,t)=\cn(k,t) \dn(k,t),
    \label{eq:6:16}
    \\
    \label{eq:6:17}
    &\begin{aligned}
      \left.
        \frac{d}{dt}
      \right|_{t=0}
      \sn(k,t)
      &=\cn(k,0) \dn(k,0)
      \\
      &=\cos(\am(k,0))
      \sqrt{1-k^2\sn^2(k,0)}=1,
    \end{aligned}
  \end{align}
  that is, $y'(0)>0$. We have (\ref{item:proposition:6:2:1}).
  \par
  For (\ref{item:proposition:6:2:2}), 
  $\am(k,t)$ is real valued if $t\in \R$ by definition of $\am(k,t)$, 
  and then $\sn(k,t)$ and $\cn(k,t)$ are also real valued if $t\in \R$. 
  On the other hand, since
  \begin{displaymath}
    \sn(k,\i x)=
    \i \frac{\sn(k',x)}{\cn(k',x)}
    \quad (k'=\sqrt{1-k^2}),
  \end{displaymath}
  the function 
  $\sn(k,t)$ is pure imaginary valued if $t\in \i \R$. 
  \par
  For (\ref{item:proposition:6:2:3}) and (\ref{item:proposition:6:2:4}), 
  write 
  $K'=- \tau K$ with $\tau\in \C$. 
  Then 
  $q \defeq e^{\i \pi \tau} =e^{-\i \pi (K'/K)}$ 
  can be written by using some series of real numbers, 
  $\{a_{\ell}\}_{\ell=0}^{\infty}$, 
  as 
  \begin{displaymath}
    q^{1/4}
    =
    \left(
      \frac{k}{4}
    \right)^{1/2}
    \left(
      \sum_{\ell=0}^{\infty}a_{\ell} k^{2 \ell}
    \right). 
  \end{displaymath}
  Thus, $q^{1/4}\in \i^{1/2} \R$ when $k=\i$, 
  which implies that $q$ is a negative real number.  
  Thus, it holds that $K'/K=1$. 
  It is known that all the poles of
  $\sn(k,x)$ are $2n K+\i (2m+1)K'$, 
  and by $K=K'$, 
  $\sn(k,x)$ has poles 
  on the lines through the origin 
  with angles $\pm \pi/4$. The other properties are well known. 
\end{proof}
By Proposition \ref{proposition:6:2}, we have 
\begin{proposition}
  \label{proposition:6:3}
  For every positive integers 
  $A$ and $C$, and a real number $a$, 
  all the solutions of both the ordinary differential equations 
  \begin{align}
    v'(t)&=\sqrt{A v(t)^4+C},\quad v(0)=a,
    \label{eq:6:18}
    \\
    v'(t)&=\sqrt{A v(t)^4-C},\quad v(0)=a, 
    \quad(\text{with } A a^4>C), 
    \label{eq:6:19}
  \end{align}
  are explosive within finite time. 
  That is, there exist positive real numbers $T_0>0$ and $T_1>0$ 
  depending on $A$, $C$ and $a$ such that 
  the existence intervals of solutions of (\ref{eq:6:18}) or (\ref{eq:6:19}) are $(-T_0,T_1)$. 
\end{proposition}
\begin{proof}
  Let $y(t) \defeq \sn(\i,t)$, 
  and 
  $w(t)\defeq - \i^{3/2}y(\i^{1/2} t)$.  
  Then, we have 
  \begin{math}
  w'
  =
  -\i^{3/2+1/2} y'
  =
  y'(\i^{1/2}t)
  \end{math}
  and also 
  \begin{displaymath}
    w'(t)^2
    =
    y'(\i^{1/2}t)^2
    =
    1-y(\i^{1/2} t)^4
    =
    1+w(t)^4
  \end{displaymath}
  since 
  \begin{math}
    w(t)^4=(-\i^{3/2})^4y(\i^{1/2}t)^4=\i^6y(\i^{1/2}t)^4=-y(\i^{1/2}t)^4.
  \end{math}
  Thus, 
  \begin{displaymath}
    w(t) 
    \defeq 
    -\i^{3/2} \sn(\i,\i^{1/2} t)
  \end{displaymath}
  is a solution of 
  \begin{displaymath}
    (w')^2=1+w^4.
  \end{displaymath}
  \par
  By the same way, 
  if we put 
  \begin{math}
    z(t) \defeq \i y(\i t), 
  \end{math}
  then 
  \begin{displaymath}
    (z')^2=\i^2 (y')^2=-(1-y^4)=y^4-1=z^4-1. 
  \end{displaymath}
  Thus, 
  \begin{displaymath}
    z(t) \defeq \i \sn(\i,\i t)
  \end{displaymath}
  is a solution of 
  \begin{displaymath}
    (z')^2=z^4-1.
  \end{displaymath}
  \par
  Therefore, any solution of (\ref{eq:6:18}) or (\ref{eq:6:19}) can be obtained by 
  $v(t) := k w(\ell t + t_0)$ or
  $v(t) := k z(\ell t + t_0)$ 
  for some constants $k$, $\ell > 0$ and some $t_0 \in \R$, where
  \begin{displaymath}
    w(t)
    =
    -\i^{3/2}
    \sn(\i,\i^{1/2} t), 
    \quad 
    z(t)=\i \sn(\i,\i t). 
  \end{displaymath}
  By Proposition \ref{proposition:6:2} (\ref{item:proposition:6:2:4}), 
  both the obtained solutions have poles, 
  so that solutions of (\ref{eq:6:18}) and (\ref{eq:6:19})
  are explosive at finite time. 
\end{proof}
\begin{remark}
  \label{remark:6:4}
  Every solution of 
  \begin{displaymath}
    (v')^2=1-v^4,\quad \vert v(0)\vert^4<1 
  \end{displaymath}
  exists on the whole line $t\in \R$. 
  This fact follows from the fact that the poles of $\sn(k,t)$ do not exist 
  on the whole real line $\R$ 
  (cf. Proposition \ref{proposition:6:2}).  
\end{remark}
The following lemma plays essential roles in the existence and non-existence result of global solutions of (\ref{eq:0}), 
and shows behavior of the energy of solution of (\ref{eq:7:1}).
\begin{lemma}
  \label{lemma:7:2}
  Let $u$ be a solution of ODE
  \begin{displaymath}
    u''(t) = A u(t) u'(t) + B (u(t))^3, 
  \end{displaymath}
  and we define $e$ and $g_k$ as
  \begin{displaymath}
    \begin{aligned}
      e(u(t)) &= \frac{1}{2}(u'(t))^2 - \frac{B}{4}(u(t))^4, 
      \\
      g_k(u(t)) &= u'(t) + k (u(t))^2, 
    \end{aligned}
  \end{displaymath}
  Moreover we assume $k$ is a real solution of 
  \begin{equation}
    \label{eq:7:2:0}
    2k^2 + k A - B = 0, 
  \end{equation}
  then, we have
  \begin{align}
    \label{eq:7:2:1}
    \frac{d}{dt} e(u(t)) 
    &=Au(t)(u'(t))^2, 
    \\
    \label{eq:7:2:2}
    \frac{d}{dt} g_k(u(t)) 
    &=
    (A + 2k)u(t) g_k(u(t)).
  \end{align}
  Moreover we may write
  \begin{equation}
    \label{eq:7:2:3}
    g_k(u(t)) 
    =
    g_k(u(0)) 
    \exp
    \left( (A + 2k)
      \int_0^t u(s)\,ds
    \right).
  \end{equation}
\end{lemma}
\begin{proof}
  Differentiating $e(u(t))$ by $t$, 
  we have
  \begin{displaymath}
    \frac{d}{dt}e(u(t)) = u' u''- B u^3 u'= A u(u')^2,
  \end{displaymath}
  so we have (\ref{eq:7:2:1}), immediately.
  Differentiating $g_k(u(t))$ by $t$, 
  we have
  \begin{displaymath}
    \frac{d}{dt}g_k(u(t)) 
    = 
    u''+ 2k u u'
    = 
    u((A + 2k) u' + Bu^2).
  \end{displaymath}
  If $k$ satisfies $2k^2 + A k - B = 0$, 
  we have (\ref{eq:7:2:2}).
\end{proof}
\section{Non-existence and existence of global solutions of the ODE}
\subsection{Main result}
In this section, we will show 
\begin{theorem}
  \label{theorem:7:1}
  Let $m\geq 3$. 
  Then, we have 
  \begin{enumerate}
  \item 
    \label{item:theorem:7:1:1}
    In the case $m\geq 5$, 
    there exists no $C^{\infty}$ global solution $u$ of (\ref{eq:5:6}) on the whole real line $\R$.
  \item 
    \label{item:theorem:7:1:2}
    In the case of $m=4$, 
    every solution $u$ of (\ref{eq:5:6}) is of the form 
    $u(t)=-b \tanh(b t + c)$ 
    for constants $b$ and $c$.
  \item 
    \label{item:theorem:7:1:3}
    In the case of $m=3$, every solution $u$ (\ref{eq:5:6}) 
    with $u(0)=0$ and $u'(0) \not= 0$ is a global bounded solution on the whole line $t\in\R$. 
  \end{enumerate}
\end{theorem}
The proof of Theorem \ref{theorem:7:1} is very long, 
so we should divide Theorem \ref{theorem:7:1} into several theorems as follows:  
The part (\ref{item:theorem:7:1:1}) of Theorem \ref{theorem:7:1} consists of three theorems, 
Theorem \ref{theorem:case8:2},
\ref{theorem:case7:1} and \ref{theorem:case9:1}.
And (\ref{item:theorem:7:1:2}) in Theorem \ref{theorem:7:1} corresponds to Theorem \ref{theorem:case4:1}, 
and (\ref{item:theorem:7:1:3}) in Theorem \ref{theorem:7:1} corresponds to Theorem \ref{theorem:case3:2}.
\par\vspace{\baselineskip}\par
First, we write the ODE (\ref{eq:5:6}) as 
\begin{equation} 
  \label{eq:7:1}
  u''=A u u'+B u^3, 
\end{equation}
where the relations of values or signs of 
$A=-\frac{m-8}{m-2}$ and $B=\frac{2(m-4)}{(m-2)^2}$ are given in the following table: 
\begin{center}
  \begin{tabular}{l | c | c | c | c | c}
    \noalign{\hrule height0.8pt}
    \hfil {} & $m=3$ & $m=4$ & $m=5,6,7$ & $m=8$ & $m\geq 9$\\ \hline
    $A$ & $+$ & $+$ & $+$ & $0$ & $-$\\
    $B$ & $-$ & $0$ & $+$ & $+$ & $+$ \\
    \noalign{\hrule height 0.8pt}
  \end{tabular}
\end{center}
In the case of $m = 3$, we also have $A^2 + 8B > 0$.
\subsection{The case of $A=0$ and $B>0$ $(m=8)$}
In this case, due to (\ref{eq:7:2:1}) of Lemma \ref{lemma:7:2}, we have immediately 
\begin{lemma}
  \label{lemma:case8:1}
  Assume that $A=0$ and $B>0$.
  If $u$ is a solution of (\ref{eq:7:1}), then 
  $e(u(t))$  is constant along the solution $u$, 
  that is, 
  \begin{equation}
    \label{eq:case8:1:1}
    u'(t)^2-\frac{B}{2}u(t)^4=u'(0)^2-\frac{B}{2} u(0)^4 \defqe 2e_0.
  \end{equation}
\end{lemma}
Then, we obtain 
\begin{theorem}
  \label{theorem:case8:2}
  In the case that $A=0$ and $B>0$, the equation (\ref{eq:7:1})
  has no global solutions defined on the whole line $\R$
  except only the trivial solution $u(t)\equiv 0$. 
  Hence the equation (\ref{eq:0}) has no global solutions on $\R$ 
  except only the trivial solutions $f(t) \equiv C$.
\end{theorem}
\begin{proof}
  Let $I = [0, T)$ be a maximal interval to exists the solution of (\ref{eq:7:1}) 
  with the initial value $u(0)$, $u'(0)$.
  Assume $u(0) = 0$ and $u'(0) > 0$ ($u(0) = 0$ and $u'(0) < 0$), 
  then there exists a positive number $\delta > 0$ 
  such that $u(t) > 0$ and $u'(t) > 0$ ($u(t) < 0$ and $u'(t) < 0$) 
  for $t \in (0, \delta)$, respectively.
  On the other hand, assume $u(0) > 0$ and $u'(0) = 0$ ($u(0) < 0$ and $u'(0) = 0$), 
  then there exists a positive number $\delta > 0$ 
  such that $u(t) > 0$ and $u'(t) > 0$ ($u(t) < 0$ and $u'(t) < 0$) for $t \in (0, \delta)$, respectively, 
  since $u''(t) = B (u(t))^3$ and $B > 0$.
  Hence, we may assume that 
  $u(0) u'(0) \not= 0$ and 
  $u(t)$ ($u'(t)$) has the same sign as $u(0)$ ($u'(0)$) for any $t \in I$, respectively.
  \par
  First we assume that $e_0 \not= 0$ and $u'(0) \not= 0$.
  In this case, by (\ref{eq:case8:1:1}), 
  we may show that $u$ satisfies the ODE:
  \begin{equation}
    \label{eq:case8:1:2}
    \begin{aligned}
      &u'(t) 
      = +\sqrt{\frac{B}{2}(u(t))^4 + e_0}, \quad \text{ if } u'(0) > 0, 
      \\
      &u'(t) 
      = -\sqrt{\frac{B}{2}(u(t))^4 + e_0}, \quad \text{ if } u'(0) < 0, 
    \end{aligned}
  \end{equation}
  with
  \begin{math}
    \frac{B}{2}(u(0))^2 + e_0 = (u'(0))^2 > 0.
  \end{math}
  By Proposition \ref{proposition:6:3}, 
  the maximal interval to exists the solution of (\ref{eq:case8:1:2}) 
  $I = [0, T)$ is finite.
  \par
  Next, we assume that $e_0 = 0$ and $u'(0) \not= 0$.
  In this case, $u$ satisfies
  \begin{equation}
    \label{eq:case8:1:3}
    (u'(t))^2= \frac{B}{2} (u(t))^4.
  \end{equation}
  The ODE (\ref{eq:case8:1:3}) is easily solved, 
  and the solution is
  \begin{displaymath}
    u(t) = \frac{u(0)}{1 \pm \sqrt{(B/2)} u(0) t}.
  \end{displaymath}
  We note that the solution blows up in finite time unless $u(0) = 0$, 
  implying $u \equiv 0$.
  Finally, by the definition of $e_0$, 
  $e_0 = u'(0) = 0$ implies $u(0) = 0$ and $u \equiv 0$.
\end{proof}

\begin{figure}[htbp]
  \centering
  \input{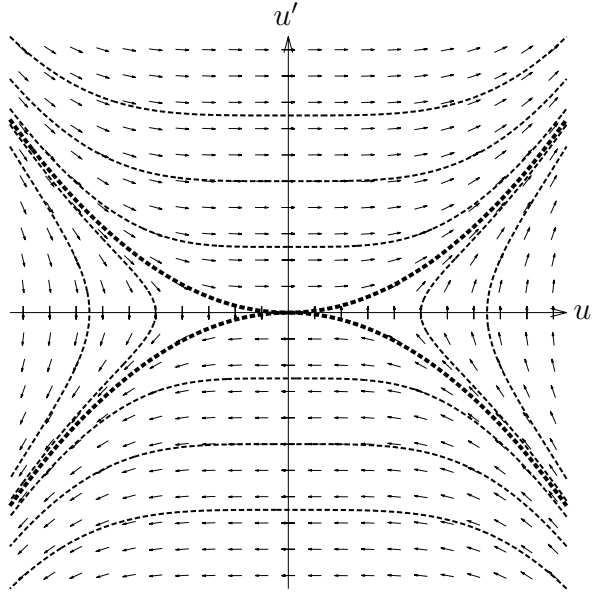}
  \caption{
    The case of $m = 8$.
    \newline 
    All dashed curves are trajectories of blow-up solutions. 
    Thick dashed curves are $u' = k_\pm u^2$.
  }
  \label{fig:case_of_8:1}
\end{figure}
\subsection{The case of $A>0$ and $B>0$ ($5 \le m \le 7$)}
\label{section:case_of_7}
\begin{proposition}
  \label{proposition:case7:4}
  Under the condition $A>0$ and $B>0$, 
  solutions of ODE (\ref{eq:7:1}) with initial values $u'(0) \ge 0$ 
  blow up in finite time.
\end{proposition}
\begin{proof}
  First, we assume $u(0) > 0$ and $u'(0) > 0$.
  Let $I = [0, T)$ be the maximal interval of existence for the solution $u$ that satisfies $u(t) > 0$, $u'(t) > 0$.
  Since $u''(t) = A u'(t) u(t) + B (u(t))^3 > 0$, 
  for any $t \in I$, we obtain $u(t) > 0$ and $u'(t) > 0$.
  Therefore, by Lemma \ref{lemma:7:2}, $e(u(t))$ is monotone increasing.
  That is, $e(u(t)) > e(u(0)) \defqe e_0$ for $t \in I$, 
  and the solution $u$ satisfies
  \begin{displaymath}
    u'(t) > \sqrt{\frac{B}{4}(u(t))^4 + e_0}.
  \end{displaymath}
  Hence the solution $u$ blows up in finite time, by Lemma \ref{lemma:6:1} and Proposition \ref{proposition:6:3}.
  \par
  In case of $u(0) = 0$ and $u'(0) > 0$, 
  since there exists a positive number $\delta > 0$ such that $u(t) > 0$, $u'(t) > 0$ for $t \in (0, \delta)$, 
  we may assume $u(t) > 0$ and $u'(0) > 0$, and the solution blows up in finite time.
  \par
  In case of $u(0) > 0$ and $u'(0) = 0$, 
  since $u''(0) = B(u(0))^3 > 0$, 
  there exists a positive number $\delta > 0$ such that $u(t) > 0$, $u'(t) > 0$ for $t \in (0, \delta)$.
  Therefore we may also assume $u(t) > 0$ and $u'(0) > 0$, and the solution blows up in finite time.
  \par
  Finally, in case of $u(0) < 0$ and $u'(0) \ge 0$, 
  define $v(t) = -u(-t)$, then $v$ also satisfies the ODE (\ref{eq:7:1}) with same $A$ and $B$, 
  and we have $v(0) > 0$, $v'(0) \ge 0$.
  Therefore $v$ also blows up in finite time.
\end{proof}
If $A > 0$ and $B > 0$, the quadratic equation (\ref{eq:7:2:0}) has two real solutions 
$k_-$ and $k_+$ satisfying $k_- < 0 < k_+$, and $A_{k-} < 0 < A_{k+}$, 
where $A_k = A + 2k$.
\begin{proposition}
  \label{proposition:case7:5}
  Under the condition $A>0$ and $B>0$, 
  solutions of ODE (\ref{eq:7:1}) with initial values $u'(0) < 0$ and $g_{k+}(u(0)) \le 0$
  blow up in finite time.
\end{proposition}
\begin{proof}
  If $u(0) = 0$ and $g_{k+}(u(0)) = 0$, then $u'(0) = 0$, 
  hence we may assume either 
  $u(0)$ 
  or $g_{k+}(u(0))$ is not zero.
  If $u(0)$ satisfies $g_{k+}(u(0)) = 0$, then, by (\ref{eq:7:2:3}), 
  we obtain $g_k(u(t)) \equiv 0$, i.e., 
  \begin{equation}
    \label{eq:case7:5:1}
    u'(t) = -k_+ (u(t))^2, \quad t > 0, \quad k_+ > 0.
  \end{equation}
  The ODE (\ref{eq:case7:5:1}) is easily solved, and 
  the solution is 
  \begin{displaymath}
    u(t) = \frac{u(0)}{1 + k_+ u(0) t}.
  \end{displaymath}
  Therefore, if $u(0) > 0$, then the solution blows up in finite time $T = -1/(k_+ u(0))$.
  If $u(0) < 0$, consider the backward solution (i.e., consider $v(t) = -u(-t)$), 
  then we also obtain similar result.
  \par
  Next we assume $u(0) < 0$ and $g_{k+}(u(0)) < 0$.
  Let $I = [0, T)$ be the maximal interval of existence for the solution $u$ that satisfies 
  $u(t) < 0$, $u'(t) < 0$ and $g_{k+}(u(t)) < 0$.
  Assume there exists $T > 0$ such that
  $u$ satisfies either
  $u(T) = 0$, 
  $u'(T) = 0$ 
  or 
  $g_{k+}(u(T)) = 0$.
  Since $u'(t) < 0$ for $t \in I$, 
  $u$ is monotone decreasing, 
  and then we obtain $u(T) < u(0) < 0$.
  Moreover, since $u(t) < 0$ for $t \in I$, 
  we also obtain $g_{k_+}(u(t)) < 0$ for $t \in I$, by (\ref{eq:7:2:3}).
  Thus we have $g_{k_+}(u(T)) < 0$, 
  which implies $u'(T) < - k_+ (u(T))^2 < 0$.
  Therefore 
  $u(t) < 0$, $u'(t) < 0$ and $g_{k+}(u(t)) < 0$
  hold provided the solution exists.
  \par
  Since for any $t > 0$, $g_{k+}(u(t)) < 0$, 
  we obtain that
  \begin{displaymath}
    u'(t) < -k_+ (u(t))^2, \quad t > 0, \quad k_+ > 0, \quad u(0) < 0.
  \end{displaymath}
  Therefore, by Lemma \ref{lemma:6:1}.
  the solution $u$ satisfies
  \begin{math}
    u(t) < \frac{u(0)}{1 + k_+ u(0) t}, 
  \end{math}
  and blows up within $T < -1/(k_+u(0))$.
  \par
  If $u(0) = 0$, since $u'(0) < 0$, 
  there exists a positive number $\delta > 0$ such that
  $u(t) < 0$, $u'(t) < 0$, $g_{k_+}(u(t)) < 0$ for any $t \in (0, \delta)$.
  Hence we may prove a blowing up phenomena within finite time for this case.
  \par
  Finally, in case of $u(0) > 0$, 
  define $v(t) = -u(-t)$, we may apply the above arguments.
\end{proof}
\begin{proposition}
  \label{proposition:case7:6}
  Under the condition $A>0$ and $B>0$, 
  solutions of ODE (\ref{eq:7:1}) with initial values $u'(0) < 0$ and $g_{k+}(u(0)) > 0$
  blow up in finite time.
\end{proposition}
\begin{proof}
  If $u(0) > 0$ and $u'(0) < 0$, 
  considering $v(t) = -u(-t)$, then we have 
  $v(0) < 0$, 
  $v'(0) < 0$ and $g_{k+}(v(0)) > 0$
  so we may assume $u(0) < 0$ without loss of generality.
  \par
  By similar arguments in Proposition \ref{proposition:case7:5}, 
  $u(t) < 0$, $u'(t) < 0$ and $g_{k+}(u(t)) > 0$
  hold provided the solution exists.
  \par
  Since $g_{k+}(u(t)) < g_{k+}(u(0)) \defqe g_0$ provided that the solution exists, 
  we obtain 
  \begin{displaymath}
    u'(t) < g_0 -k_+ (u(t))^2, 
    \quad
    g_0 > 0, 
    \quad
    k_+ > 0, 
    \quad
    u(0) < 0.
  \end{displaymath}
  The ODE $v'(t) = g_0 - k_+ (v(t))^2$ is well-known logistic equation, 
  and the solution $v$ blows up to $-\infty$ within positive finite time 
  provided $g_0 - k_+(v(0))^2 < 0$.
  Since $g_0 - k_+(u(0))^2 = u'(0) < 0$, 
  by Lemma \ref{lemma:6:1},
  $u$ blows up in finite time.
\end{proof}
\begin{theorem}
  \label{theorem:case7:1}
  In the case that $A>0$ and $B>0$, the equation (\ref{eq:7:1})
  has no global solutions defined on the whole line $\R$
  except only the trivial solution $u(t)\equiv 0$. 
  Hence the equation (\ref{eq:0}) has no global solutions on $\R$ 
  except only the trivial solutions $f(t) \equiv C$.
\end{theorem}
\begin{proof}
  In case of $u'(0) \ge 0$, 
  use Proposition \ref{proposition:case7:4}, 
  in case of $u'(0) < 0$ and $g_{k+}(u(0)) \le 0$, 
  use Proposition \ref{proposition:case7:5}, 
  and
  in case of $u'(0) < 0$ and $g_{k+}(u(0)) > 0$, 
  use Proposition \ref{proposition:case7:6}, 
  we obtain the claim of Theorem \ref{theorem:case7:1}.
\end{proof}

\begin{figure}[htbp]
  \centering
  \input{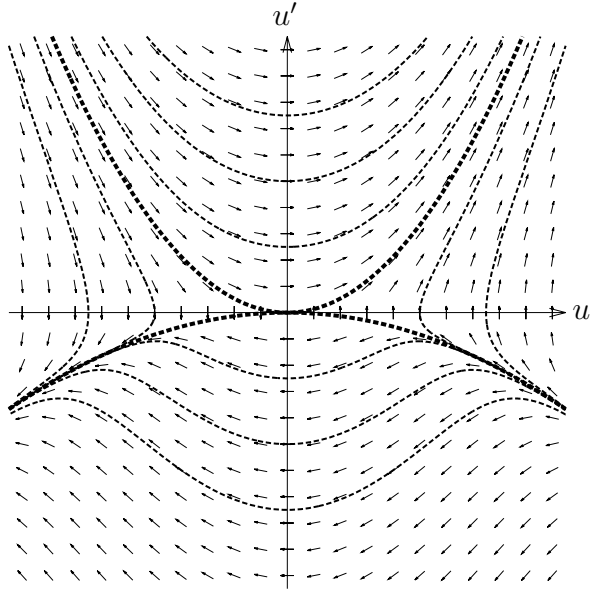}
  \caption{
    The case of $m = 5$.
  }
  \label{fig:case_of_7:1}
\end{figure}
\subsection{The case of $A<0$ and $B>0$ ($m \ge 9$)}
\begin{theorem}
  \label{theorem:case9:1}
  In the case that $A<0$ and $B>0$, the equation (\ref{eq:7:1})
  has no global solutions defined on the whole line $\R$
  except only the trivial solution $u(t)\equiv 0$. 
  Hence the equation (\ref{eq:0}) has no global solutions on $\R$ 
  except only the trivial solutions $f(t) \equiv C$.
\end{theorem}
\begin{proof}
  Let $u$ be a solution of (\ref{eq:7:1}), 
  and $v(t) = u(-t)$, then 
  $v$ satisfies $v''(t) = -Av(t)v'(t) + B(v(t))^3$.
  Therefore the claim is easily obtained by Theorem \ref{theorem:case7:1}.
\end{proof}
\begin{figure}[htbp]
  \centering
  \input{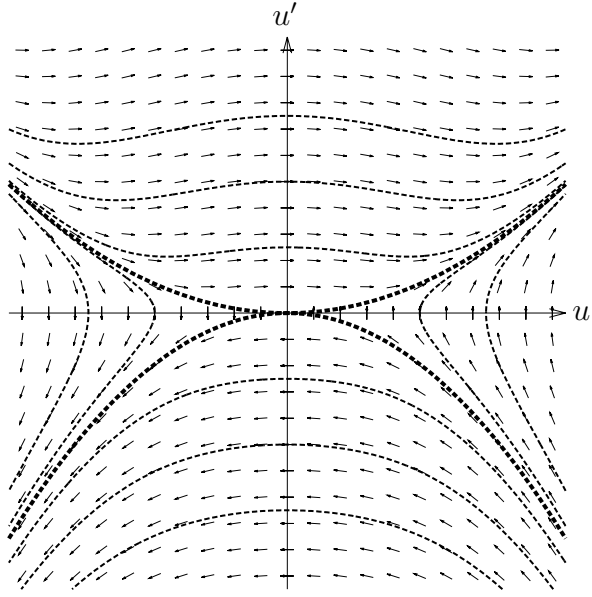}
  \caption{
    The case of $m = 11$.
  }
  \label{fig:case_of_9:1}
\end{figure}
\subsection{The case $A>0$ and $B=0$ $(m=4)$} 
Rewriting (\ref{eq:7:1}) as 
$u''(t) = (A/2) (u(t)^2)'$ and
integrating this equation, 
we obtain the following proposition.
\begin{theorem}
  \label{theorem:case4:1}
  In case of $A>0$ and $B=0$, 
  all global solutions of (\ref{eq:7:1})
  are given by 
  $u(t) = -b \tanh(b t + c)$, 
  where $b$ and $c$ are constants.
  Hence all global solutions of (\ref{eq:0})
  are given by 
  $f(x) = a/\cosh(b x + c)$, where $a > 0$.
\end{theorem}
\begin{proof}
  Integrating $u''(t) = (A/2) (u(t)^2)'$, we have
  $u'(t) = (A/2)(u(t)^2 + C)$, 
  where $C \defeq (2/A) (u'(0) - (A/2) u(0)^2)$.
  Assume $C < 0$ and $|u(0)| < \sqrt{|C|}$, 
  then 
  we easily obtain that
  \begin{equation}
    \label{eq:proposition:case4:1:1}
    u(t) 
    =
    - \sqrt{|C|} \tanh\left((A_C/2)t - \arctanh\left(\frac{u(0)}{\sqrt{|C|}}\right)\right), 
  \end{equation}
  where $A_C \defeq A\sqrt{|C|}$, and 
  (\ref{eq:proposition:case4:1:1}) is defined on whole line $\R$.
  In case of $m = 4$, $A$ is equal to $2$, and $A_C = \sqrt{|C|}$, 
  hence we may write 
  \begin{equation}
    \label{eq:proposition:case4:1:2}
    u(t) 
    =
    - b \tanh(b t + c).
  \end{equation}
  Since the solution $f$ of (\ref{eq:0}) is given by
  \begin{math}
    f(x) = \exp\left(\int_0^x u(t)\,dt\right), 
  \end{math}
  by (\ref{eq:proposition:case4:1:2}), we obtain 
  \begin{displaymath}
    f(x) 
    = 
    \frac{a}{\cosh(b x + c)}.
  \end{displaymath}
  If $C < 0$ and $|u(0)| > \sqrt{|C|}$, 
  the solution is given by 
  \begin{displaymath}
    u(t) = -\sqrt{|C|} \frac{\sqrt{|C|}\tanh(A_C t) - u(0)}{\sqrt{|C|} - u(0)\tanh(A_C t)}.
  \end{displaymath}
  However, 
  the denominator attains its zero at 
  $t = (1/A_C)\arctanh(\sqrt{|C|}/u(0))$, 
  hence this type of solution is not globally defined.
  If $C > 0$, 
  the solution is given by 
  \begin{displaymath}
    u(t) = \sqrt{C} \frac{\sqrt{C}\tan(A_C t) + u(0)}{\sqrt{C} - u(0)\tan(A_Ct)}, 
  \end{displaymath}
  hence this type of solution is not globally defined.
  In case of $C = 0$ and $u'(0)^2 \not= 0$, 
  the solution is given by 
  \begin{math}
    u(t) = \frac{-u(0)}{(A/2) u(0) t - 1}, 
  \end{math}
  hence this type of solution is also not globally defined.
\end{proof}
\begin{remark}
  \label{remark:case4:2}
  In case of $u'(0)^2 = 0$ and $u(0)^2 = -C$, the solution is stationary.
\end{remark}
\begin{figure}[htbp]
  \centering
  \input{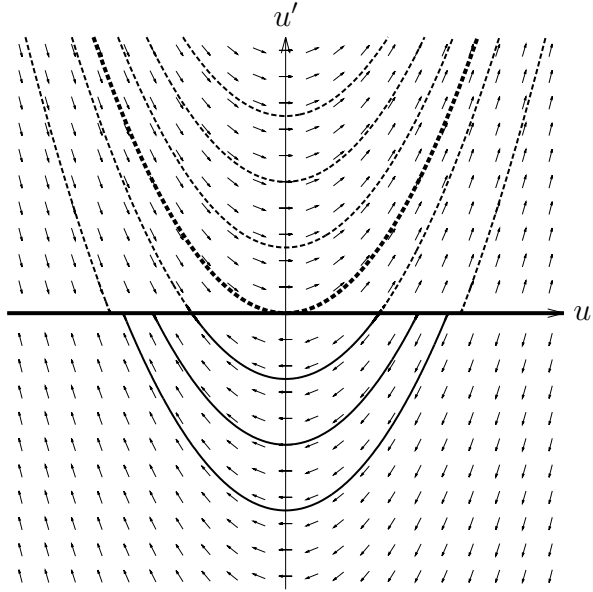}
  \caption{
    The case of $m = 4$.
    \newline 
    All solid curves are trajectories of global solutions. 
    Each point on the thick solid line $u' = 0$ is stationary.
  }
  \label{fig:case_of_4:1}
\end{figure}
\subsection{The case $A>0$, $B<0$ and $A^2 + 8B \ge 0$ $(m=3)$} 
\begin{proposition}
  \label{proposition:case3:1}
  Under the condition $A>0$ and $B<0$, 
  assume the initial value satisfies $u(0) \le 0$ and $u'(0) < 0$.
  Then there exists $T > 0$ such that 
  the solution of ODE (\ref{eq:7:1}) with the initial value exists on $[0, T]$ 
  and it satisfies 
  $u(T) < 0$ and $u'(T) = 0$.
\end{proposition}
\begin{proof}
  If $u(0) = 0$ and $u'(0) < 0$, 
  there exists a positive number $\delta > 0$ such that
  $u(t) < 0$ and $u'(t) < 0$ for $t \in (0, \delta)$, 
  hence we may assume $u(0) < 0$, $u'(0) < 0$ without loss of generality.
  Let $I = [0, T)$ be the maximal interval of existence for the solution $u$ that satisfies $u(t) < 0$ and $u'(t) < 0$.
  If $t \in I$, then $u$ is monotone decreasing and 
  $u''(t) = Au(t)u'(t) + B(u(t))^3 > B (u(t))^3 > B(u(0))^3 > 0$, 
  hence $u'$ is monotone increasing and $e(u(t))$ is monotone decreasing.
  \par
  Assume $t \in I$, 
  we have $(1/2)(u'(t))^2 + (|B|/4)(u(t))^4 = e(u(t)) < e(u(0))$, 
  therefore $u$ and $u'$ is bounded.
  Moreover for any $t \in I$, 
  by using $u''(t) > B(u(t))^3 > B(u(0))^3$, 
  we obtain
  \begin{displaymath}
    0 
    > u'(t) 
    > u'(0) + t B (u(0))^3.
  \end{displaymath}
  Since $u'(0) < 0$ and $B (u(0))^3 > 0$, 
  there exists $0 < T < -u'(0)/(B(u(0))^3)$, such that
  $u'(T) = 0$. 
  Moreover $u'(t) < 0$ for any $t \in (0, T)$, 
  we obtain $u(T) \le u(t) < u(0)$.
\end{proof}
\begin{proposition}
  \label{proposition:case3:2}
  Under the condition $A>0$, $B < 0$ and $A^2 + 8B \ge 0$, 
  let $k_1$ and $k_2$ are real solutions of (\ref{eq:7:2:0}) with $k_2 \le k_1 < 0$, 
  and assume the initial value satisfies $u(0) \le 0$, $u'(0) \ge 0$ 
  and $g_{k_1}(u(0)) < 0$.
  Then the solution of ODE (\ref{eq:7:1}) with the initial value exists on $[0, \infty)$ and 
  it satisfies 
  $u(t) \to 0$ and $u'(t) \to 0$ as $t \to \infty$.
\end{proposition}
\begin{proof}
  If $u'(0) = 0$, 
  since $u''(0) = A u'(0)u(0) + B(u(0))^3 = B (u(0))^3 > 0$, 
  there exists a positive number $\delta > 0$ such that
  $u(t) < 0$, $u'(t) > 0$, $g_{k_1}(u(t)) < 0$ for $t \in (0, \delta)$.
  Hence we may assume $u(0) < 0$, $u'(0) > 0$ and $g_{k_1}(u(0)) < 0$ without loss of generality.
  \par
  Let $I = [0, T)$ be the maximal interval of existence for the solution $u$ that satisfies $u(t) < 0$, $u'(t) > 0$ and $g_{k_1}(u(t)) < 0$.
  By Lemma \ref{lemma:7:2}, $g_{k_1}(u(t))$ is always negative while the solution exists.
  If $u(T) = 0$, 
  then we obtain $g_{k_1}(u(T)) = u'(T) + k_1(u(T))^2 = u'(T) < 0$.
  This contradicts to $u'(t) > 0$ for $t \in I$.
  Therefore, $u(t) < 0$ while the solution exists.
  If $u'(T_1) = 0$ for some $T_1 > 0$, 
  then 
  we obtain $u''(T_1) = B(u(T_1))^3 \ge 0$, 
  since $u''(t) = A u(t) u'(t) + B(u(t))^3$, $u'(T_1) = 0$, $u(T_1) \le 0$ and $B < 0$.
  Hence $u'$ is non-decreasing on $(T_1 - \delta, T_1)$ for some $\delta > 0$.
  This contradicts to $u'(T_1 - \delta) > 0$.
  Therefore 
  $u(t) < 0$, $u'(t) > 0$ and $g_{k_1}(u(t)) < 0$
  hold provided the solution exists.
  \par
  Now assume $t \in I$, by $g_{k_1}(u(t)) < 0$, 
  we have
  \begin{displaymath}
    u'(t) < -k_1 (u(t))^2, 
    \quad
    u(0) < 0.
  \end{displaymath}
  By Lemma \ref{lemma:6:1}, we obtain
  \begin{displaymath}
    u(0) < u(t) < \frac{u(0)}{1 + k_1 u(0) t} < 0, 
  \end{displaymath}
  and the solution exists on $[0, \infty)$, 
  since $k_1u(0) > 0$.
  In particular, 
  we obtain
  \begin{displaymath}
    \int_0^t u(s)\,ds 
    =
    \int_0^t \frac{u(0)}{1 + k_1 u(0) s}\,ds
    =
    \frac{1}{k_1} \log\left|1 + k_1 u(0) t\right| \to -\infty, 
  \end{displaymath}
  as $t\to\infty$.
  Hence we obtain that
  \begin{displaymath}
    g_\infty 
    =
    \lim_{t\to\infty} g_{k_1}(u(t))
    =
    g_{k_1}(u(0))
    \lim_{t\to\infty} 
    \exp
    \left(
      A_{k_1}
      \int_0^t u(s)\,ds
    \right)
    = 0.
  \end{displaymath}
  On the other hand, 
  since $u(t) < 0$ and $u$ is monotone increasing, 
  there exists $u_\infty \le 0$ such that $u(t) \to u_\infty$, 
  and $u'(t) \to 0$ ($t \to \infty$).
  Therefore using $g_\infty = k_1 u_\infty^2$, we obtain $u_\infty = 0$.
\end{proof}
\begin{theorem}
  \label{theorem:case3:2}
  In the case that $A>0$, $B < 0$ and $A^2 + 8B \ge 0$, 
  there exist global solutions of (\ref{eq:7:1}) on whole real line $\R$ 
  satisfying $u(t) \to 0$, $u'(t) \to 0$ ($t \to \pm \infty$).
  Hence there exist positive global solutions of (\ref{eq:0}) on $\R$ 
  satisfying $f(t) \to C$ ($t\to\pm\infty$).
\end{theorem}
\begin{proof}
  Let the initial condition satisfy $u(0) = 0$ and $u'(0) < 0$, 
  then, by Proposition \ref{proposition:case3:1}, 
  there exists $T > 0$ such that 
  the solution exists on $[0, T]$ and 
  it satisfies $u(T) < 0$ and $u'(T) = 0$.
  \par
  If $u(T) < 0$ and $u'(T) = 0$, 
  then we have $g_{k_1}(u(T)) = k_1(u(T))^2 < 0$.
  Hence, 
  by using time-shift $t \mapsto t - T$, and Proposition \ref{proposition:case3:2}, 
  the solution satisfies $u(T) < 0$ and $u'(T) = 0$ extends to $[T, \infty)$, 
  and it satisfies $u(t) \to 0$ and $u'(t) \to 0$ ($t \to \infty$).
  Therefore, 
  we obtain the solution on $[0, \infty)$ with the initial value $u(0) = 0$ and
  $u'(0) < 0$.
  \par
  Consider the backward solution of the ODE, 
  we may easily prove that 
  the solution extends on the whole real line $\R$, 
  and it satisfies $u(t) \to 0$ and $u'(t) \to 0$ ($t \to -\infty$).
\end{proof}
\begin{theorem}
  \label{theorem:case3:3}
  Under the condition $A>0$, $B < 0$ and $A^2 + 8B \ge 0$, 
  the equation (\ref{eq:7:1}) admits no non-trivial periodic solutions.
  Hence the equation (\ref{eq:0}) admins no positive non-trivial periodic solutions.
\end{theorem}
\begin{proof}
  Assume that the equation (\ref{eq:7:1}) admits a non-trivial periodic solution $u$.
  If $u$ satisfies $u'(t) > 0$ or $u'(t) < 0$ for any $t \in \R$, 
  then $u$ must be a monotone function which never occurs because $u$ is periodic.
  Hence we obtain that
  there exists $T > 0$ such that $u'(T) = 0$.
  Moreover if there exists $T > 0$ such that $u'(T) = u(T) = 0$, 
  then $u$ is trivial.
  Therefore, we may assume that
  there exists $T > 0$ such that $u'(T) = 0$ and $u(T) \not= 0$.
  If $u(T) > 0$, considering $v(t) = -u(-t)$, 
  we may assume $u(T) < 0$ without loss of generality.
  By Proposition \ref{proposition:case3:2}, 
  $u$ should satisfy $u(t) \to 0$ as $t \to \infty$.
  Hence $u$ is not periodic.
\end{proof}
\begin{figure}[htbp]
  \centering
  \input{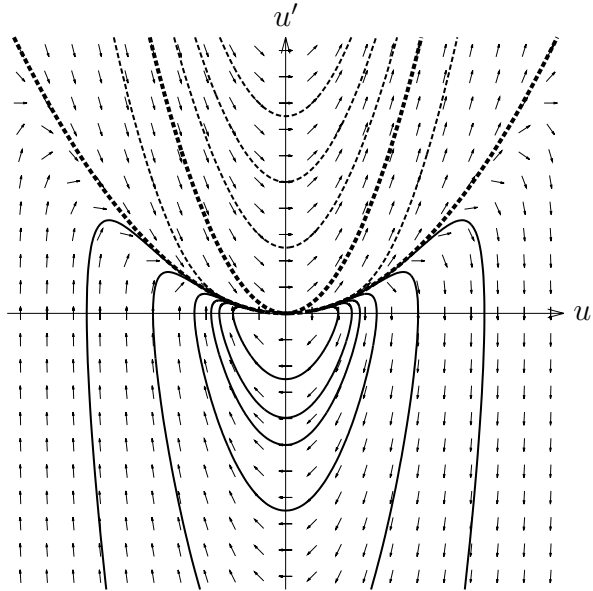}
  \caption{The case of $m = 3$.}
  \label{fig:case_of_3:1}
\end{figure}
\begin{figure}[htbp]
  \centering
  \input{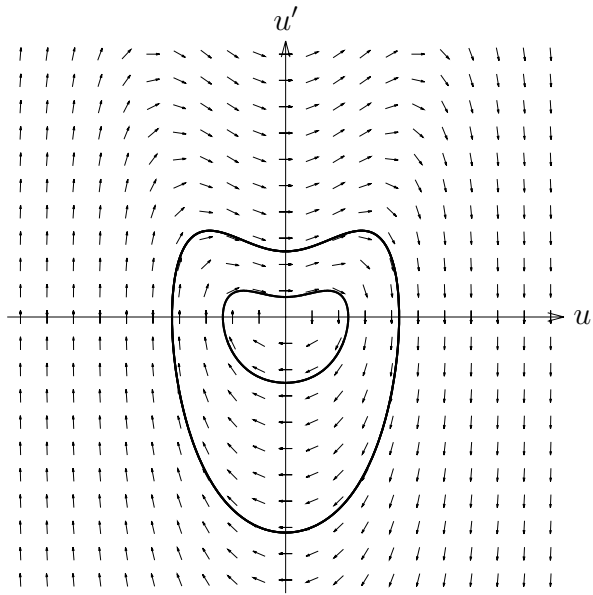}
  \caption{The case of $A^2 + 8B < 0$ ($A = 2$, $B = -4$).}
  \label{fig:case_of_3:2}
\end{figure}
\subsection{Remarks}
By above arguments, 
we also obtain that if the quadratic equation (\ref{eq:7:2:0}) has real solutions $k_1$, $k_2$ and
they have same sign, then there exist global bounded solutions of (\ref{eq:7:1}), 
and if they have different sign, 
then there exist no global solutions of (\ref{eq:7:1}).
\par
On contrary, we conjecture that 
under the condition $A^2 + 8 B < 0$ (i.e., (\ref{eq:7:2:0}) has no real solutions), 
all solutions of (\ref{eq:7:1}) are periodic.
If $A = 0$, the equation (\ref{eq:7:1}) can be written $u'' = -\kappa^4 u^3$, 
Solutions of this equation are written as $u(t) = C \sn(i, C\kappa (t+t_0))$ by using Jacobi's sn, 
and it is well-known that they are periodic (cf. Remark \ref{remark:6:4}).
For $A \not= 0$, 
numerical experiments support this conjecture (see Figure \ref{fig:case_of_3:2}).
However set $A = (8-m)/(m-2)$, $B = 2(m-4)/(m-2)^2$, 
then $A^2 + 8B = m^2/(m-2)^2$, 
therefore, 
there are no real numbers $m$ satisfying the condition $A^2 + 8 B < 0$.
\par
Finally, we note that we use classical Runge-Kutta method (Figures \ref{fig:case_of_8:1}, \ref{fig:case_of_7:1}, \ref{fig:case_of_9:1}, \ref{fig:case_of_4:1} and \ref{fig:case_of_3:1}), 
and Gauss method of order 6 (Figure \ref{fig:case_of_3:2}) as numerical integrators (cf. \cite{Heirer}).
\section{Biharmonic maps between product Riemannian manifolds}
Finally, we give nice applications. 
Let us consider the product Riemannian manifolds, 
$M \defeq \R \times \Sigma^{m-1}$, and 
$N \defeq \R\times P$, respectively, 
where $\R$ is a line
with the standard Riemannian metric $g_1$, 
$\Sigma^{m-1}$ is an $(m-1)$-dimensional manifold with a Riemannian metric $g_2$ $(m=3,4)$, 
and $P$ is a manifold with Riemannian metric $h_2$, respectively. 
Let us take the product Riemannian metrics 
$g=g_1+g_2$ on $M$, 
and $h=g_1+h_2$ on $N$, respectively. 
\par
Then, for every smooth map 
$\varphi=(\varphi_1,\varphi_2) \colon M\rightarrow N$, 
with $\varphi_1 \colon \R \rightarrow \R$, 
and $\varphi_2 \colon \Sigma^2\rightarrow P$, 
the tension field $\tau(\varphi)$ is given as 
\begin{displaymath}
  \tau(\varphi)=(\tau(\varphi_1),\tau(\varphi_2))\in \Gamma(\varphi^{-1}TN)=
  \Gamma(\varphi_1^{-1}T{\mathbb R}\times 
  \varphi_2^{-1}TN).
\end{displaymath}
Thus, $\varphi$ is harmonic if and only if both 
(1) $\varphi_1 \colon (\R,g_1)\rightarrow (\R,g_1)$ is harmonic, 
and  
(2) $\varphi_2 \colon (\Sigma^{m-1},g_2)\rightarrow (P,h_2)$ is harmonic.  
Notice that all the harmonic maps $\varphi_1 \colon (\R,g_1)\rightarrow (\R,g_1)$ are affine functions  
$\R \ni x\mapsto ax+b\in \R$ for some constants $a$ and $b$. 
\par
Now we
define a conformal Riemannian metric 
$\widetilde{g}=\widetilde{f}^{2/(m-2)}g$ 
with $\widetilde{f}(x,y)=f(t)$ ($t=x\in \R$, $y\in \Sigma^{m-1}$). 
\par
Then, we can easily calculate that 
\begin{align*}
  \nabla^g f
  &=f'\,\frac{\partial}{\partial t},\\
  \varphi_{\ast}(\nabla^g f)
  &=\varphi_1{}_{\ast}(f'\frac{\partial}{\partial t})=a f'\frac{\partial}{\partial t},\\
  \Delta^g f
  &=-f'',\\
  \overline{\nabla}_{\nabla^g f}
  \varphi_{\ast}(\nabla^g f)
  &=a f''f' \frac{\partial}{\partial t},\\
  J_g(\varphi_{\ast}(\nabla^g f))
  &=-a f''' \frac{\partial}{\partial t}. 
\end{align*}
For a harmonic map 
$\varphi=(\varphi_1,\varphi_2) \colon (M,g)=(\R\times\Sigma^{m-1},g)\rightarrow (N,h)= (\R\times P,h)$, 
it holds that 
$\varphi \colon (M,\widetilde{g}) \rightarrow (N,h)$  
is harmonic if and only if 
$\varphi_{\ast}(\nabla^g f)=a f'\frac{\partial}{\partial t}=0$ 
if and only if 
$f(t)$ is constant in $t=x$ or $\varphi_1$ is a constant. 
\par
Therefore, we obtain 
\begin{lemma}
  The above mapping 
  $\varphi\colon (M,\widetilde{g})\rightarrow (N,h)$ is a biharmonic map if and only if the mapping 
  $\varphi_1\colon \R \rightarrow \R$ is a constant or it satisfies the ODE (\ref{eq:0}). 
\end{lemma}
Thus, we obtain the following theorem which answers our Problem in the Section \ref{sec:4}, 
in the case of the (non-compact) product Riemannian manifolds, and the product harmonic maps.  
\begin{theorem}
  \label{theorem:9:1}
  For every harmonic map $\varphi \colon (\Sigma^{m-1},g)\rightarrow (P,h)$, 
  let us define 
  $\widetilde{\varphi} \colon \R \times \Sigma^{m-1}\ni (x,y)\mapsto 
  (ax+b,\varphi(y))\in \R \times P$ $(m=3,4)$, where 
  $a$ and $b$ are constants. 
  Then, 
  \begin{enumerate}
  \item 
    In the case $m=3$, 
    the mapping $\widetilde{\varphi} \colon (\R \times \Sigma^2,\widetilde{f}^2 g)\rightarrow (\R \times P,h)$ 
    is biharmonic, but not harmonic if $a\not=0$.
  \item 
    In the case $m=4$, the mapping  
    $\widetilde{\varphi} \colon (\R \times \Sigma^3,\frac{1}{\cosh x} g)
    \rightarrow (\R \times P,h)$ is biharmonic, but not harmonic if $a\not=0$.   
  \end{enumerate}
\end{theorem}

\end{document}